\documentclass[12pt,reqno]{amsart}
\usepackage{}
\usepackage{amsfonts}
\usepackage{a4wide}
\allowdisplaybreaks
\numberwithin{equation}{section}

\newtheorem{theorem}{Theorem}[section]
\newtheorem{proposition}[theorem]{Proposition}

\newtheorem{lemma}[theorem]{Lemma}

\theoremstyle{definition}

\newcommand{\R}{\mathbb{R}}
\newcommand{\dis}{\displaystyle}

\begin{document}

\title
 [Standing waves for nonlinear Schr\"odinger equation with rotation]
 {Multiplicity, asymptotics and stability of standing waves for nonlinear Schr\"odinger equation with rotation}

{\let\thefootnote\relax \footnotetext{This work was supported by National Natural Science Foundation of China (Grant Nos. 11901147, 11771166), the Fundamental Research Funds for the Central Universities of China (Grant No. JZ2020HGTB0030) and the excellent doctorial dissertation cultivation from Central China Normal University (Grant No. 2019YBZZ064).}}


\maketitle
\begin{center}
\author{Xiao Luo}
\footnote{Email addresses: luoxiao@hfut.edu.cn (X. Luo).}
\author{Tao Yang}
\footnote{Email addresses: yangt@mails.ccnu.edu.cn (T. Yang).}
\end{center}

%

\begin{center}
\address {1 School of Mathematics, Hefei University of Technology, Hefei, 230009, P. R. China}

\address {2 School of Mathematics and Statistics, Central China Normal University, Wuhan, 430079, P. R. China}
\end{center}

\maketitle

\begin{abstract}
In this article, we study the multiplicity, asymptotics and stability of standing waves  with prescribed mass $c>0$ for nonlinear Schr\"odinger equation with rotation in the mass-supercritical regime arising in Bose-Einstein condensation. Under suitable restriction on the
rotation frequency, by searching critical points of the corresponding energy functional on the mass-sphere, we obtain a local minimizer $u_c$ and a mountain pass solution $\hat{u}_c$. 
Furthermore, we show that $u_c$ is a ground state for small mass $c>0$ and describe a mass collapse behavior of the minimizers as $c\to 0$, while $\hat{u}_c$ is an excited state. Finally, we prove that the standing wave associated with $u_c$ is stable. Notice that the pioneering works \cite{aMsC,shYZ} imply that finite time blow-up of solutions to this model occurred in the mass-supercritical setting, therefore, we in the present paper obtain a new stability result. The main contribution of this paper is to extend the main results in \cite{JeSp,gYlW} concerning the same model from mass-subcritical and mass-critical regimes to mass-supercritical regime, where the physically most relevant case is covered.


\end{abstract}
\maketitle




\section{Introduction and Main Results}
In this paper, we study the multiplicity, asymptotics and stability of standing waves with prescribed mass for the nonlinear Schr\"odinger equation with rotation
\begin{equation}\label{1.1}
    \left\{
 \renewcommand{\arraystretch}{1.1}
 \begin{array}{ll}
  i \partial_{t} \psi=-\frac{1}{2} \Delta \psi+V(x)\psi-\Omega \cdot L \psi-a|\psi|^{p-2} \psi, \quad \ (t,x)\in \R^+ \times \R^N,
  \vspace {.25cm}\\
 \psi(0,x)=\psi_{0}(x),
 \end{array}
 \right.
\end{equation}
where $a\!>\!0$, $N\!=\!2,3$, $V(x)\!=\!\frac{|x|^2}{2}$ and $2\!+\!\frac{4}{N}\!\leq\!p\!<\!2^*\!:=\!\frac{2N}{(N-2)^{+}}$.
The rotation term $\Omega \cdot L$ reads
$$\Omega \cdot L:=-i \Omega \cdot (x \wedge \nabla)=-i (\Omega \wedge x) \cdot \nabla=-i|\Omega|\left(x_{1} \partial_{x_{2}}-x_{2} \partial_{x_{1}}\right),$$
where $i\!=\!\sqrt{-1}$, $\Omega\!=\!(0,0,|\Omega|) \!\in \!\R^3$ is a given angular velocity vector, $L\!=\!-i x \wedge \nabla$ is the quantum mechanical angular momentum operator and $\wedge$ is the wedge product of the two vectors. 

Problem \eqref{1.1} with $p=4$ arises in Bose-Einstein condensation (BEC), which describes the quantum effects in macro scope. Physically, the BEC is set into rotation by a stirring potential, which is usually induced by a laser \cite{kMaD,kMFD,MAnd,Stoc} (see also \cite{WbAo} for numerical simulations). The appearance of quantum vortices in rapidly rotating BEC has been studied by \cite{FtAl,BaHw,CaWl,NrCo,Dalf,Fete,tMUd} and the references therein. In the mean-field regime, rotating BEC can be accurately described by the nonlinear Schr\"odinger equation \eqref{1.1} (see \cite{eBgr,EiRi,eiRi,MeCa}). The power term in \eqref{1.1} describes the mean-field self-interaction of the condensate particles and the parameter $a\!>\!0$ ($a\!<\!0$) characterizes the strengthen of attractive (repulsive) interactions between the cold atoms. Vortices are believed to be unstable for $a\!>\!0$ (see \cite{CaWl,ColA,tMUd}), but they form stable lattice configurations for $a\!<\!0$ (see \cite{FtAl,NrCo,Fete}). We mainly consider the existence and stability of standing waves of \eqref{1.1} in the general case $2\!+\!\frac{4}{N}\!\leq\!p\!<\!2^*$, where the physically most relevant case $p\!=\!4$ is covered.


Throughout this paper, we denote the norm of ${L^p}({\R^N})$ by ${\left\| u \right\|_p}: = {({\int_{{\R^N}} {\left| u \right|} ^p})^{\frac{1}{p}}}$ for any $1 \le p < \infty $.
Our working space $\Sigma$ is defined as
\[\Sigma:=\Big\{u\in {H^1}({\R^N},\mathbb{C}): \int_{{\R^N}} {{{\left| x \right|}^2}{{\left| u \right|}^2}}<+\infty\Big\},\]
which is a Hilbert space  with the inner product and norm
\[(u,v)_\Sigma: = \text{Re}\int_{{\R^N}} ({\nabla u\nabla \bar{v}+|x|^2u\bar{v}}  + {u\bar{v}})dx,~~~~
{\left\| u \right\|_\Sigma} := (\|u\|_{\dot{\Sigma}}^2+ \left\| u \right\|_2^2)^{\frac{1}{2}},\]
where $\|u\|_{\dot{\Sigma}}^2:=\left\| {\nabla u} \right\|_2^2 +\left\| xu \right\|_2^2$,``Re" stays for the real part and $\bar{v}$ denotes the conjugate of $v$. We use ``$\rightarrow$" and ``$\rightharpoonup$" respectively to denote the strong and weak convergence in the related function spaces. $C$ will denote a positive constant unless specified. $o_{n}(1)$ and $O_{n}(1)$ mean that $|o_{n}(1)|\to 0$ as $n\to+\infty$ and $|O_{n}(1)|\leq C$ as $n\to+\infty$, respectively. $\mathbb{R}$ and $\mathbb{C}$ denote the sets of real and complex numbers respectively.

A standing wave of \eqref{1.1} with a prescribed mass $c\!>\!0$ is a solution having the form $\psi (t,x) \!=\! e^{-i \omega t} u(x)$ for some $(u,\omega)\!\in\! \Sigma \!\times\! \R$ such that $\|u\|_2^2\!=\!\|\psi \|_2^2\!=\!c$ and $u$ weakly solves
\begin{align}\label{1.3}
\left(-\frac{1}{2} \Delta+\frac{1}{2}|x|^2-(\Omega \cdot L)\right) u-a|u|^{p-2} u=\omega u, ~~~~~~~~x\in \R^N
\end{align}
in the following sense
\[  \text{Re} \Big[\frac{1}{2}  \int_{{\R^N}} {\nabla u\nabla \bar{\varphi} }  + \int_{{\R^N}} {\frac{1}{2}|x|^2u\bar{\varphi} }-\int_{{\R^N}} \bar{\varphi}(\Omega \cdot L) u-a\int_{{\R^N}}|u|^{p-2} u\bar{\varphi}-\omega \int_{{\R^N}} {u\bar{\varphi} }\Big]=0,~~~~\forall\varphi \in \Sigma.\]
To study the time-dependent equation \eqref{1.1}, we shall concern firstly the stationary equation \eqref{1.3}. Physicists usually call this type of solution $u$ a ``normalized solution" to \eqref{1.3}. This fact implies that $\omega$ cannot be determined a priori, but is part of the unknown. Normalized solutions to \eqref{1.3} can be obtained by searching critical points of the energy functional
\begin{align}\label{1.4}
 \dis I(u):= \frac{1}{2}\int_{\mathbb{R}^{N}}|\nabla u|^{2}+\frac{1}{2}\int_{\mathbb{R}^{N}}|x|^2|u|^{2}
 -\frac{2a}{p}\int_{\mathbb{R}^{N}}|u|^{p}-\int_{\mathbb{R}^{N}} \bar{u}(\Omega \cdot L) u d x
\end{align}
on the constraint
\begin{align}\label{1.5}
    S(c) = \{ u \in \Sigma:{\left\| u \right\|_{{2}}^2} = c\}
\end{align}
with Lagrange multipliers $\omega$. 

Recently, J. Arbunich et al. in \cite{JeSp} concerned the existence,  stability and instability properties of the standing waves to \eqref{1.1} with $\dis V(x)=\sum_{j=1}^{N}\frac{\gamma_{j}^{2} x_{j}^{2}}{2}$. They assumed that
\begin{align*}
&(i)~~~~a<0,~~~~|\Omega|<\min_{1\leq j\leq N}  \gamma_{j},~~~~2< p < 2^*; ~~~~(ii)~~~~a>0,~~~~|\Omega|<\min_{1\leq j\leq N}  \gamma_{j},~~~~2< p < 2+\frac{4}{N}.
\end{align*}
Under these hypotheses, the authors in \cite{JeSp} studied the global minimization problem
\begin{align}\label{1.6}
E(c): = \mathop {\inf }\limits_{u \in S(c)} I(u).
\end{align}
They proved the relative compactness of any minimizing sequence for \eqref{1.6} and hence obtained the existence of minimizers and stability of  ${\mathcal{M} }_c$, where
\begin{align}\label{1.8}
{{\mathcal{M} }_c}: = \{ u \in S(c):I(u) =E(c)\}.
\end{align}
Whether the minimizer of \eqref{1.6} is radially symmetric or not is still unknown. When $|\Omega|=0$ and $V(x)\!=\!\frac{|x|^2}{2}$, the results of \cite{Mark,MhiO} indicated that the minimizer of \eqref{1.6} is radially symmetric. In \cite{EiRi,eiRi,IgMv}, a symmetry breaking result for the energy minimizers was proved for $|\Omega|$ above a certain critical
speed ${\Omega}_{crit}>0$. Furthermore, an estimate for ${\Omega}_{crit}$ in $N=2$ can be found in \cite{IgMv}.

Notice that $\bar{p}:=\!2\!+\!\frac{4}{N}$ is the $L^2$-critical or mass-critical  exponent for problem \eqref{1.6} since $E(c)\!>\!-\infty$ if $p\!\in\!(2,\bar{p})$ and $E(c)\!=\!-\infty$ if $p\!\in\!(\bar{p},2^*)$. Indeed, for fixed $u \!\in\! S(c)$, we have $u_\tau(x)\!=\!\tau^{\frac{N}{2}}u(\tau x)\!\in\! S(c)$ and
\begin{align*}
\dis I(u_\tau)\!=\! \frac{\tau^2}{2}   {\| {\nabla u} \|}_2^2 \! +\! \frac{1}{2\tau^2}{\| xu \|}_2^2\!-\! \frac{{2a}}{p}\tau^{p\delta_p}  {{\| u \|}_p^p}\!-\!\int_{{\R^N}} \bar{u}(\Omega \cdot L) u\!\to\!-\infty~~~~\mbox{as}~~~~\tau\to+\infty,
\end{align*}
where
\begin{align}\label{delta_p}
\delta_p=\frac{N(p-2)}{2p}.
\end{align}

Y. J. Guo et al. in \cite{gYlW} considered the two-dimensional attractive BEC (i.e. $N\!=\!2$, $a\!>\!0$ and $p\!=\!4$ in \eqref{1.3}) in a general trap $V(x)$ satisfying $0 \!\leq\! V(x) \!\in\! L_{loc}^{\infty}(\R^{2})$ and $\dis \varliminf_{|x| \rightarrow \infty} \frac{V(x)}{|x|^{2}}\!>\!0$, which falls in the $L^2$-critical case. They proved that there exists $\Omega^{*}\!>\!0$ and $a^*\!>\!0$ such that
$(i)$~~~~if $0\!\leq\!|\Omega|\!<\!\Omega^{*}$ and $0\!\leq\! a\!<\!a^*$, there exists at least one global minimizer;
$(ii)$~~~~if $0\!\leq\!|\Omega|\!<\!\Omega^{*}$ and $a\!\geq\! a^*$, there is no global minimizer; $(iii)$~~~~if $|\Omega|\!>\!\Omega^{*}$ and $a\!>\!0$, there is no global minimizer.
The authors also analyzed the limit behavior and mass concentration of the global minimizers as a $a\nearrow a^*$ if $0<|\Omega|<\Omega^{*}$.

%


\vspace {.25cm}

To our best knowledge, the existence and stability of standing waves to \eqref{1.1} with $\bar{p}<p<2^*$ is still unknown. Much attention should be paid to this case since it contains the physically most relevant case $p=4$, $N=3$.

Since $E(c)=-\infty$ for $p\in(\bar{p},2^*)$, the global minimization method adopted in \eqref{1.6} does not work. Furthermore, due to the existence of the trapping potential $\frac{1}{2}|x|^2$, it seems not applicable to find a critical point of $I\left| {_{S(c)}} \right.$ by minimizing $I$ on a constructed submanifold of $S(c)$ as \cite{ll} does. Motivated by \cite{BNlv,bj}, 
we study a local minimization problem: for any given $r>0$, define
\begin{align}\label{1.9}
m_c^r: = \mathop {\inf }\limits_{u \in S(c)\cap B(r)} I(u),
\end{align}
where
$$B(r)=\Big\{u\in \Sigma: \|u\|_{\dot{\Sigma}}^2=\left\| {\nabla u} \right\|_2^2 +\left\| xu \right\|_2^2\leq r\Big\}.$$
For any fixed $r\!>\!0$, it is clear that $m_c^r\!>\!-\infty$ if $S(c)\cap B(r)\!\neq\! \emptyset$. We will claim that $S(c)\cap B(r)\!\neq\! \emptyset$ (see Lemma \ref{lem2.2+}) and $m_c^r$ is achieved. Once the claim is true, we have
\begin{align}\label{1.9++}
{\mathcal{M}_c^r}: = \{ u \in S(c)\cap B(r):I(u) = m_c^r\}\neq \emptyset.
\end{align}
After excluding the possibility of the minimizers locating on the boundary of $S(c)\cap B(r)$, then the minimizer of $m_c^r$ is indeed a critical point of $I\left| {_{S(c)}} \right.$ as well as a normalized solution to \eqref{1.3}. The main results in this aspect are stated as follows.

%
%
%
%
%

\begin{theorem} \label{the1.5}
Let $a \!>\!0$, $N\!=\!2,3$, $2\!+\!\frac{4}{N}\!\leq\!p\!<\!2^*$ and $0\!<\!|\Omega|\!<\!1$. For any fixed $r\!>\!0$, we could find some $c_0:=\!c_0(r,a,|\Omega|)\!>\!0$ such that for any $c\!<\!c_0$, there exist $(u_{c},\omega_{c})\!\in\! \Sigma \!\times\! \R $ such that $u_c\!\in\! {\mathcal{M}_c^r}$ and $u_{c}$ weakly solves \eqref{1.3} with $\omega=\omega_{c}$. Furthermore,
$$  N\Big(\frac{1-{|\Omega|}^2}{2(1+3|\Omega|)}-a \mathcal{C}^p_{N,p} r^{\frac{p\delta_p-2}{2}} c^{\frac{p(1-\delta_p)}{2}}\Big)\!\leq\!\omega_c\!<\!
\frac{N}{2}  $$
and
$$  \mathop {\sup }\limits_{u \in {\mathcal{M}_c^r}} \left\| {u - l_0{\psi _0}} \right\|_{ \Sigma }^2=O(c+c^{\frac{p(1-\delta_p)}{2}}),$$ 
where $\delta_p\!=\!\frac{N(p-2)}{2p}$, $\mathcal{C}_{N,p}$ is some positive constant, ${\psi _0}$ is the unique normalized positive eigenvector of the harmonic oscillator $-\Delta\!+\!|x|^2$ and $l_0\!=\!\int_{{\R^N}}{u\psi _0}$.  \end{theorem}

Next, we show that $u_c $ is a normalized ground state if $c>0$ is sufficiently small and concern the asymptotic behavior of $u_c$ as $c\to 0^+$. Following \cite{bj}, we say that $u_c\in S(c)$ is a normalized ground state solution to problem \eqref{1.3} if
\[ I^{'}|_{S(c)}(u_c)=0 \text{ and } I(u_c)=\inf\{I(u): u\in S(c), \text{ }I^{'}|_{S(c)}(u)=0\}.\]


\begin{theorem} \label{Asymtotic}
Assume that $a \!>\!0$, $N\!=\!2,3$, $2\!+\!\frac{4}{N}\!<\!p\!<\!2^*$ and $0\!<\!|\Omega|\!<\!\sqrt{1\!-\!(\frac{2}{p\delta_p})^2}$. Let $(u_{c},\omega_{c})\!\in\! {\mathcal{M}_c^r} \!\times\! \R $ be given by Theorem \ref{the1.5} and $I(u_c)=m_c^r$. Then, $u_{c}$ is a normalized ground state to \eqref{1.3} provided $c>0$ is  sufficiently small. Furthermore, $u_{c}\to 0$ in $\dot{\Sigma}$ as $c\to 0^+$, $\mathop {\lim }\limits_{c\to 0^+}\frac{m_c^r}{c}=\mathop {\lim }\limits_{c\to 0^+}\omega_c=\omega$ for some $\omega\in[ \frac{(1-{|\Omega|}^2)N}{2(1+3|\Omega|)}, \frac{N}{2}]$ and
\[
 \dis\mathop {\lim }\limits_{c\to 0^+}\frac{\|\nabla u_{c}\|_2^2 \!-\! \dis\int_{\mathbb{R}^{N}}\bar{u}_{c}(\Omega \cdot L) u_{c}   }{c}=\mathop {\lim }\limits_{c\to 0^+} \frac{\|x u_{c}\|_2^2 \!-\!  \dis\int_{\mathbb{R}^{N}}\bar{u}_{c}(\Omega \cdot L) u_{c}}{c}=\omega.\]
\end{theorem}

\noindent \textbf{Remark 1.1} Theorem \ref{the1.5} implies that the standing wave $\psi_c(t,x) \!=\! e^{-i \omega_c t} u_c(x)$ of \eqref{1.1} behaves like the first eigenvector of the harmonic oscillator for small $c>0$. Theorem \ref{Asymtotic} describes a mass collapse behavior of the minimizers $u_{c} \in {\mathcal{M}_c^r}$. It indicates that $u_{c}\to 0$ in $\dot{\Sigma}$ with $\|\nabla u_{c}\|_2^2\!-\!\int_{\mathbb{R}^{N}}\bar{u}_{c}(\Omega \cdot L) u_{c}$ and $\|x u_{c}\|_2^2\!-\!\int_{\mathbb{R}^{N}}\bar{u}_{c}(\Omega \cdot L) u_{c}$ converging to $0$ at the same rate, and the corresponding frequency $\omega_c$ converges to some $\omega$ as $c\!\to\!0^+$. Due to the existence of the rotation term in \eqref{1.1}, the limit $\omega$ is inaccurate. When the rotation frequency $|\Omega|\!$ vanishes, we could get
$$\mathop {\lim }\limits_{c\to 0^+}\frac{m_c^r}{c}=\mathop {\lim }\limits_{c\to 0^+}\omega_c=\frac{N}{2},~~~~~~~~\mathop {\lim }\limits_{c\to 0^+} \frac{\|\nabla u_{c}\|_2^2}{c} =\mathop {\lim }\limits_{c\to 0^+}       \frac{\|x u_{c}\|_2^2}{c} =\frac{N}{2},$$
where we notice that $\frac{N}{2}$ is the first eigenvalue of $-\frac{1}{2}\Delta\!+\!\frac{1}{2}|x|^2$. \\



P. Antonelli et al. in \cite{aMsC} proved the local well-posedness of \eqref{1.1} in $\Sigma$ (See Lemma 3.1 of \cite{aMsC}), which states that for any $u_0 \!\in \!\Sigma$, there exists a $T\!>\!0$ and a unique solution $u\!\in\! C([0,T),\Sigma)$ of \eqref{1.1} with $u(0,\cdot)\!=\!u_0$. In addition, they proved that the mass and energy are preserved for all $t\!\in\! [0,T)$, where either $T\!=\!+\infty$ or $T\!<\!+\infty$ and $\lim_{t\to T^-}\|\nabla u\|_2\!=\!+\infty$. We say that a set $Y\subset\Sigma$ is \textbf{stable} under the flow associated with problem \eqref{1.1} if for any $\varepsilon>0$, there exists $\delta>0$ such that for any $u_0 \in \Sigma$ satisfying
\[\text{dist}_\Sigma(u_0,Y)<\delta,\]
the solution $u(t,\cdot)$ of problem \eqref{1.1} with $u(0,\cdot)=u_0$ satisfies
\[\dis \sup_{t\in [0,T)}\text{dist}_\Sigma(u(t,\cdot),Y)<\varepsilon,\]
where $T$ is the existence time for $u(t,\cdot)$. With these preliminaries, we then study the stability of ${\mathcal{M} }^r_c$.

\begin{theorem} \label{stability}
Let $a \!>\!0$, $N\!=\!2,3$, $2\!+\!\frac{4}{N}\!<\!p\!<\!2^*$ and $0\!<\!|\Omega|\!<\!1$. Then, ${\mathcal{M}_c^r}$ is stable under the flow corresponding to problem \eqref{1.1}.
\end{theorem}

\noindent \textbf{Remark 1.2}
In \cite{aMsC}, P. Antonelli et al. proved the global existence of solutions to \eqref{1.1} with $\dis V(x)=\sum_{j=1}^{N}\frac{\gamma_{j}^{2} x_{j}^{2}}{2}$ provided either $a<0$ and $2< p < 2^*$ or $a>0$ and $2< p < 2+\frac{4}{N}$. On the contrary, finite time blow-up of the solutions to \eqref{1.1} occurred in two cases: $(i)$ $a>0$, $(\Omega \cdot L)V(x)=0$ and $2+\frac{4}{N}\leq p<2^*$; $(ii)$ $a>0$, $(\Omega \cdot L)V(x)\not=0$, $|\Omega|<{\gamma}$ and $\dis 2+\frac{4}{N}\sqrt{\frac{{\gamma}^2}{{\gamma}^2-|\Omega|^2}}\leq p<2^*$, where $\dis {\gamma}:=\min_{1\leq j\leq N}  \gamma_{j}$. More recently, N. Basharat et al. in \cite{shYZ} also studied \eqref{1.1}. They obtained a sharp condition on the global existence and blowup of solutions to \eqref{1.1} for $p=\bar{p}$ and some blowup conditions for $\bar{p}<p<2^*$. Moreover, similar results were extended to  \eqref{1.1} with an inhomogeneous nonlinearity. Compared with the results in \cite{aMsC,shYZ}, we obtain a new stability result in Theorem \ref{stability}. \\

Recall that the solutions obtained in Theorem \ref{the1.5} are local minimizers of $I|_{S(c)}$ and $I|_{S(c)}$ is unbounded from below for $p\!\in\!(\bar{p},2^*)$. Motivated by \cite{bj}, by using the local minimizers in ${\mathcal{M}_c^r}$, we obtain a mountain pass critical point of $I|_{S(c)}$.
\begin{theorem} \label{the1.8}
Let $a \!>\!0$, $N\!=\!2,3$, $2\!+\!\frac{4}{N}\!<\!p\!<\!2^*$ and $0\!<\!|\Omega|\!<\!\sqrt{1\!-\!(\frac{2}{p\delta_p})^2}$. For any $c<c_0$,  there exist $(\hat{u}_c,\hat{\omega}_c)\!\in\! \Sigma \!\times\! \R $ such that $\hat{u}_c$ weakly solves \eqref{1.3} with $\omega=\hat{\omega}_c$ and $I({\hat{u}_c})>m_c^r$, where $c_0$ is defined by Theorem \ref{the1.5}.
\end{theorem}

We give the outline of the proof for our main results. Theorem \ref{the1.5} is proved by searching minimizers of $I|_{S(c)\cap B(r)}$. Once $m_c^{r}>-\infty$ is proved, each minimizing sequence of $m_c^{r}$ is bounded in $\Sigma$. Observing that $\frac{1}{2}\|u\|_{\dot{\Sigma}}^2-\int_{\mathbb{R}^{N}}\bar{u}(\Omega \cdot L) u d x$ is an equivalent norm in $\Sigma$ provided $0\!<\!|\Omega|\!<\!1$, we deduce that $I$ is weakly lower semi-continuous (see \eqref{boud2.4}). Moreover, Lemma \ref{Compact} gives the compactness of the embedding $\Sigma\hookrightarrow L^q(\R^N,\mathbb{C})$ for $q\!\in\![2,2^*)$, then the existence of minimizer to $m_c^{r}$ follows. The rest is to show that the minimizer is not on the boundary of $S(c)\cap B(r)$, then it is indeed a critical point of $I|_{S(c)}$. To this end, we find a suitable constant $c_0\!=\!c_0(r,a,|\Omega|)$ such that for $c<c_0$, it holds that
\begin{align}\label{1.9+++}
\mathop {\inf }\limits_{u \in S(c)\cap B(\nu r)} I(u)<\mathop {\inf }\limits_{u \in S(c)\cap (B(r)\setminus B(\mu r))} I(u),
\end{align}
where $\nu\!=\!\frac{1-|\Omega|}{4}$ and $\mu\!=\!\frac{1+|\Omega|}{2}$.
This local minima structure \eqref{1.9+++} guarantees that all minimizing sequences of $m_c^{r}$ shrink and results in ${\mathcal{M}_c^r}\subset B(\nu r)$, which leads to the minimizer of \eqref{1.9} is bounded away from the boundary of $S(c)\cap B(r)$.

The proof of Theorem \ref{Asymtotic} mainly comes from \cite{BNlv}. The key point is to prove that
\[ v\in S(c)~~~~~~~~\mbox{such that}~~~~~~~~I^{'}|_{S(c)}(v)=0 \text{ and } I(v)<m_c^r \Longrightarrow v\in B(r)~~~~\text{as}~~~~c\to 0.\]
However, our case is different from the case $|\Omega|\!=\!0$ since the rotation term $\int_{\mathbb{R}^{N}} \bar{v}(\Omega \cdot L) v d x$ within $I(v)$ is sign indefinite. In fact, $I^{'}|_{S(c)}(v)=0$ gives the Pohozaev identity $$Q(v):=\frac{1}{2}  {{{\| {\nabla v} \|}_2^2}}  \!-\! \frac{1}{2} {{{\| xv \|}_2^{2}}}\!-\!a\delta_p { {{\| v \|}_p^p}}\!=\!0,$$
and hence $I(v)$ can be rewrite as $I(v)=\Big(\frac{1}{2}-\frac{1}{p\delta_p} \Big) {{{\| {\nabla v} \|}_2^2}}  \!+\! \Big(\frac{1}{2}+\frac{1}{p\delta_p} \Big) {{{\| xv \|}_2^{2}}}\!-\!\int_{{\R^N}} \bar{v}(\Omega \cdot L) v$, then the extra condition $0\!<\!|\Omega|\!<\!\sqrt{1\!-\!(\frac{2}{p\delta_p})^2}$ indicates that
\[
 C\|v\|_{\dot{\Sigma}}^2 \leq I(v) <m_c^r<\frac{Nc}{2}\to 0\ \text{as}\ c\to 0
\]
for some constant $C>0$. So $v\in B(r)$ as $c\to 0$ follows.

To prove Theorem \ref{stability}, we use the fact that any minimizing sequence of $m_c^r$ is precompact and ${\mathcal{M}_c^r}\neq\emptyset$ (see the proof of Theorem \ref{the1.5}). By a contradiction argument, we obtain the stability of ${\mathcal{M}_c^r}$.

Theorem \ref{the1.8} is proved by a variant of mountain pass theorem. Let $\gamma(c)$ be the mountain pass level, we will construct a special Palai-Smale sequence $\{v_n\}$ at energy level $\gamma(c)$ with $Q(v_n)\to0$.
When $|\Omega|\!=\!0$, the property $Q(v_n)\to0$ is sufficient to derive the boundedness of $\{v_n\}$ (see \cite{bj}). However, the term $\int_{\mathbb{R}^{N}} \bar{v}_n(\Omega \cdot L) v_n d x$ within $I(v_n)$ is sign indefinite if $0\!<\!|\Omega|\!<\!1$ and we can not proceed as in \cite{bj}. Under the stronger condition $0\!<\!|\Omega|\!<\!\sqrt{1\!-\!(\frac{2}{p\delta_p})^2}$, we can prove that
\[
 C\|v_n\|_{\dot{\Sigma}}^2+o_n(1)\leq I(v_n) \leq \gamma(c)+1
\]
for some constant $C>0$. Then $\{v_n\}$ is bounded in $\Sigma$. The rest is standard as in \cite{bj}.\\

\noindent \textbf{Remark 1.3}
The conditions $0\!<\!|\Omega|\!<\!1$ in Theorem \ref{the1.5} and  $0\!<\!|\Omega|\!<\!\sqrt{1\!-\!(\frac{2}{p\delta_p})^2}$ in Theorem \ref{the1.8} are necessary. In fact, the essence of the restrictions on $|\Omega|$ is that, 
$$ \|u\|^2_{\Omega_1} \approx \|u\|_{\dot{\Sigma}}^2 ~~~~~~~~\mbox{if}~~~~~~~~0\!<\!|\Omega|\!<\!1,~~~~~~~~ \|u\|^2_{\Omega_2} \approx \|u\|_{\dot{\Sigma}}^2~~~~~~~~\mbox{if}~~~~~~~~
0\!<\!|\Omega|\!<\!\sqrt{1\!-\!(\frac{2}{p\delta_p})^2},~~~~~~~~\forall u\in\Sigma,$$
where
\[\|u\|^2_{\Omega_1}:=\! \frac{1}{2}\|u\|_{\dot{\Sigma}}^2\!-\!\int_{\mathbb{R}^{N}}\bar{u}(\Omega \cdot L) u, \|u\|^2_{\Omega_2}:=\! \Big(\frac{1}{2}\!-\!\frac{1}{p\delta_p} \Big) {{{\| {\nabla u} \|}_2^2}}  \!+\! \Big(\frac{1}{2}\!+\!\frac{1}{p\delta_p} \Big) {{{\| x{u} \|}_2^{2}}}\!-\!\int_{{\R^N}} \bar{u}(\Omega \cdot L) u,\]
and $\|\cdot\|_{\mathcal{A}} \approx \|\cdot\|_{\mathcal{B}}$ means $\|\cdot\|_{\mathcal{A}}$ and $\|\cdot\|_{\mathcal{B}}$ are two equivalent norms. As pointed out above, $0\!<\!|\Omega|\!<\!1$ guarantees $ \|\cdot\|_{\Omega_1} \!\approx\! \|\cdot\|_{\dot{\Sigma}}$ and hence the weakly lower semi-continuity of $I$ in proving Theorem \ref{the1.5}, and $0\!<\!|\Omega|\!<\!\sqrt{1\!-\!(\frac{2}{p\delta_p})^2}$ guarantees $ \|\cdot\|_{\Omega_2} \!\approx\! \|\cdot\|_{\dot{\Sigma}}$ and hence the boundedness of the corresponding Palai-Smale sequence in proving Theorem \ref{the1.8}. Alternatively, we can obtain Theorem \ref{the1.5} by studying
\begin{align*}
m_c^r: = \mathop {\inf }\limits_{u \in S(c)\cap B(r)} I(u)~~~~~~~~\mbox{for}~~~~~~~~B(r)=\Big\{u\in \Sigma: \|u\|_{\Omega_1}^2\leq r\Big\}.
\end{align*}

\noindent \textbf{Remark 1.4}
Our main results in the present paper can be extended from $\dis V(x)\!=\!\frac{|x|^2}{2}$ to $\dis V(x)= \sum_{j=1}^{N}\frac{\gamma_{j}^{2} x_{j}^{2}}{2}$, with the rotation frequency satisfying $\dis 0\!<\!|\Omega|\!<\min_{1\leq j\leq N} \gamma_{j}$. Here $\gamma_j>0$ ($j=1,\cdots,N$) is the trapping frequencies in each spatial direction, see \cite{aMsC,JeSp,shYZ} for details.\\

The paper is organized as follows. In Section 2, we present some preliminary results. In Section 3, we prove Theorems \ref{the1.5}-\ref{stability}. In Section 4, we prove Theorem \ref{the1.8}. 

\vspace{.25cm}

\section{Preliminary Results}
In this section, we give some preliminary results. Firstly, we give the  Gagliardo-Nirenberg inequality (See \cite{Wein}).

\begin{lemma} \label{lem2.1}
Let $N\!\geq\!2$ and $p\!\in\!(2,2^{*})$. Then there exists a constant $\mathcal{C}_{N,p}\!>\!0$ such that
\begin{equation} \label{equ2.2}
 ||u||_{p} \leq \mathcal{C}_{N,p} \left\|\nabla u\right\|_{2}^{\delta_p} \left\|u\right\|_{2}^{(1-\delta_p)}, \qquad \forall u \in {H}^{1}(\mathbb{R}^{N},\R)
\end{equation}
where $\mathcal{C}_{N,p}\!=\!\Big( \frac{p}{2 ||W_p||^{p-2}_{2}} \Big)^{\frac{1}{p}}$, $W_p$ is the ground state solution of
$ -\Delta W\!+\!(\frac{1}{\delta_p}-\!1)W \!=\!\frac{2}{p\delta_p}|W|^{p-2}W$ and  $\delta_p\!=\!\frac{N(p-2)}{2p}$.
\end{lemma}

\begin{lemma} (\cite{Wein}) \label{Lem2.2}
Let $|x|u$ and $\nabla u$ belong to $L^2(\R^N,\R)$. Then, $u\in L^2(\R^N,\R)$ and
$$ \|u\|_{2}^{2} \leq \frac{2}{N}\|\nabla u\|_{2}\|x u\|_{2},$$
with equality holding for the functions $u(x)=e^{-\frac{1}{2}|x|^2}$.
\end{lemma}

\noindent \textbf{Remark 2.1} Lemma \ref{lem2.1} remains true for any $u \in {H}^{1}(\mathbb{R}^{N},\mathbb{C})$ and Lemma \ref{Lem2.2} remains true for any $u\in \Sigma$ since $\big|\nabla |u|\big|\leq |\nabla u|$, see Theorem 6. 17 in \cite{ELMA}.

\begin{lemma} \label{lem2.2+}
For any $r>0$, $S(c)\cap B(r)\neq \emptyset$ iff $c\leq\frac{r}{N}$.
\end{lemma}

\begin{proof}
Let $r>0$ be fixed. For any $u\in S(c)\cap B(r)\neq \emptyset$, Lemma \ref{Lem2.2} and Remark 2.1 indicate that
$$ c=\|u\|_{2}^{2} \leq \frac{2}{N}\|\nabla u\|_{2}\|x u\|_{2}
\leq \frac{2}{N}\Big(\frac{1}{2}\|\nabla u\|_{2}^2+\frac{1}{2}\|x u\|_{2}^2\Big)
 =\frac{1}{N} \|u\|_{\dot{\Sigma}}^2 \leq \frac{r}{N}.$$
On the other hand, let $\phi(x)=e^{-\frac{1}{2}|x|^2}$ and $\psi_0=\pi^{-\frac{N}{4}}e^{-\frac{1}{2}|x|^2}$, then we have
$$ \|\nabla \phi\|^{2}_{2}=\|x \phi\|^{2}_{2}=\frac{N}{2}\|\phi\|_{2}^{2}=\frac{N}{2}\pi^{\frac{N}{2}},~~~~~~~~
\|\nabla \psi_0\|^{2}_{2}=\|x \psi_0\|^{2}_{2}=\frac{N}{2},~~~~~~~~
\|\psi_0\|_{2}^{2}=1.$$
For any $c\leq\frac{r}{N}$, we have $\sqrt{c}\psi_0\in S(c)\cap B(r)$.
\end{proof}

By Young's inequality and the fact that $\Omega \cdot L:=-i \Omega \cdot (x \wedge \nabla)=-i (\Omega \wedge x) \cdot \nabla$, we obtain the following interpolation inequality.
\begin{lemma}(\cite{JeSp}, Inequality (2.3)) \label{LTa2.2}
Let $\Omega \cdot L:=-i \Omega \cdot (x \wedge \nabla)$. For any $\varepsilon>0$, it holds that
\begin{align}\label{Gn1.2}
| \langle u,(\Omega \cdot L) u \rangle |
\leq \|(\Omega \wedge x) u \|_{2} \|\nabla u \|_{2}
\leq \frac{|\Omega|^{2}}{2 \varepsilon} \|x u \|_{2}^{2}+\frac{\varepsilon}{2}\|\nabla u \|_{2}^{2},~~~~\forall u\in \Sigma.
\end{align}
\end{lemma}
We recall the following compactness result:
\begin{lemma} (\cite{Mark,ZAnG}) \label{Compact}
For $N\!\geq\!2$ and $q\!\in\![2,2^*)$, the embedding $\Sigma\hookrightarrow L^q(\R^N,\mathbb{C})$ is compact.\\
\end{lemma}

\section{Proof of Theorems \ref{the1.5}-\ref{stability}}

In this section, we prove Theorems \ref{the1.5}-\ref{stability}. 
To begin with we show that $I|_{S(c)}$ presents a local minima structure by the previous lemmas. This fact guarantees that the minimizer of $m_c^r$ is indeed a critical point of $I|_{S(c)}$.

\begin{proposition} \label{prop2.3}
Let $a \!>\!0$, $N\!=\!2,3$, $2\!+\!\frac{4}{N}\!\leq\!p\!<\!2^*$ and $0<|\Omega|<1$. For any $r>0$, if $S(c)\cap (B(\mu r)\setminus B(\nu r))\neq\emptyset$, then there exists $c_0:=c_0(r,a,\Omega)>0$ such that for any $c<c_0$,
\begin{align}\label{2.3}
\mathop {\inf }\limits_{u \in S(c)\cap B(\nu r)} I(u)<\mathop {\inf }\limits_{u \in S(c)\cap (B(r)\setminus B(\mu r))} I(u),
\end{align}
where $\nu=\frac{1-|\Omega|}{4}$ and $\mu=\frac{1+|\Omega|}{2}$.
\end{proposition}

\begin{proof}
Let $\nu\!=\!\frac{1-|\Omega|}{4}$ and $\mu\!=\!\frac{1+|\Omega|}{2}$ be fixed, then $0\!<\!\nu\!<\!\mu\!<\!1$ and $\frac{\nu}{\mu}\!<\!\frac{1-\sqrt{|\Omega|}}{1+\sqrt{|\Omega|}}
<\frac{1-|\Omega|}{1+|\Omega|}$ as $0\!<\!|\Omega|\!<\!1$. By Lemma \ref{lem2.2+}, $S(c)\cap B(\nu r)\neq \emptyset$ iff $c\leq\frac{\nu r}{N}$.
If $S(c)\cap B(\nu r)\neq \emptyset$ and $S(c)\cap (B(r)\setminus B(\mu r))\neq\emptyset$, we will prove \eqref{2.3}. Since $\frac{\nu}{\mu}<\frac{1-|\Omega|}{1+|\Omega|}$, we can choose $\varepsilon_0\in (\frac{\mu+\nu}{\mu-\nu}|\Omega|^{2},\frac{\mu-\nu}{\mu+\nu})$ and denote
\begin{align*}
C_{*}(\Omega):=\min \Big \{ \frac{(1-\varepsilon_0)}{2}, \Big(\frac{1}{2}-\frac{|\Omega|^{2}}{2 \varepsilon_0}\Big) \Big\},~~~~ C^{*}(\Omega):=\max \Big \{ \frac{(1+\varepsilon_0)}{2}, \Big(\frac{1}{2}+\frac{|\Omega|^{2}}{2 \varepsilon_0}\Big) \Big\}.
\end{align*}
We deduce from $\frac{\nu}{\mu}\!<\!\frac{1-\sqrt{|\Omega|}}{1+\sqrt{|\Omega|}}$ that $\frac{\mu+\nu}{\mu-\nu}|\Omega|^{2} \!<\! \frac{\mu+\nu}{\mu-\nu}|\Omega| \!<\! \frac{\mu-\nu}{\mu+\nu}$. From now on, let $$\varepsilon_0=\frac{\mu+\nu}{\mu-\nu}|\Omega|
=\frac{3+|\Omega|}{1+3|\Omega|}|\Omega|$$
be fixed. Direct calculations imply that
\begin{align} \label{consbound}
 \frac{1-{|\Omega|}^2}{2(1+3|\Omega|)}=C_{*}(\Omega)\!<\!C^{*}(\Omega)
 =\frac{1+6|\Omega|+{|\Omega|}^2}{2(1+3|\Omega|)}
~~~~\mbox{and}~~~~C^{*}(\Omega)
\!<\!\frac{\mu}{\nu}C_{*}(\Omega).
\end{align}
Applying inequality \eqref{Gn1.2} in Lemma \ref{LTa2.2} with $\varepsilon\!=\!\varepsilon_0$, we have
\begin{align} \label{boud2.4}
C_{*}(\Omega) \|u\|_{\dot{\Sigma}}^2 \leq  \frac{1}{2}\|u\|_{\dot{\Sigma}}^2-\int_{\mathbb{R}^{N}}\bar{u}(\Omega \cdot L) u d x \leq C^{*}(\Omega) \|u\|_{\dot{\Sigma}}^2.
\end{align}
We see that
$ \frac{1}{2}\|u\|_{\dot{\Sigma}}^2-\int_{\mathbb{R}^{N}}\bar{u}(\Omega \cdot L) u d x$ is a new norm which is equivalent to $\|u\|_{\dot{\Sigma}}^2$. This fact is also observed by N. Basharat et al. in \cite{shYZ}.  

For any $u \in S(c)\cap (B(r)\setminus B(\mu r))$, by \eqref{equ2.2} in Lemma \ref{lem2.1} and \eqref{boud2.4}, we have
\begin{align}\label{2.4}
I(u)&= \displaystyle\frac{1}{2} [\left\| {\nabla u} \right\|_2^2 +\left\| xu \right\|_2^2]- \frac{{2a}}{p}\left\| u \right\|_p^p- \int_{\mathbb{R}^{N}}\bar{u}(\Omega \cdot L) u d x \nonumber\\
&\geq \displaystyle C_{*}(\Omega) \|u\|_{\dot{\Sigma}}^2-\frac{{2a}}{p} \mathcal{C}^p_{N,p} \left\|\nabla u\right\|_{2}^{p\delta_p} c^{\frac{p(1-\delta_p)}{2}}\geq \displaystyle C_{*}(\Omega) \|u\|_{\dot{\Sigma}}^2 -\frac{{2a}}{p} \mathcal{C}^p_{N,p} \|u\|_{\dot{\Sigma}}^{p\delta_p} c^{\frac{p(1-\delta_p)}{2}} \nonumber\\
&=\displaystyle\|u\|_{\dot{\Sigma}}^2 \Big(C_{*}(\Omega)-\frac{{2a}}{p} \mathcal{C}^p_{N,p} \|u\|_{\dot{\Sigma}}^{p\delta_p-2} c^{\frac{p(1-\delta_p)}{2}}\Big)\geq \displaystyle \mu r \Big(C_{*}(\Omega)-\frac{{2a}}{p} \mathcal{C}^p_{N,p} r^{\frac{p\delta_p-2}{2}} c^{\frac{p(1-\delta_p)}{2}}\Big) 
\end{align}
where we restrict $c<\Big[\frac{pC_{*}(\Omega)}{{2a\mathcal{C}^p_{N,p}  r^{\frac{p\delta_p-2}{2}}}} \Big]^{\frac{2}{p(1-\delta_p)}}$ in the last inequality.

On the other hand, for any $u\in { S(c)\cap B(\nu r)}$, we deduce from \eqref{boud2.4} that
\begin{align}\label{2.5}
I(u)\!=\! \displaystyle\frac{1}{2}\|u\|_{\dot{\Sigma}}^2 \displaystyle\!-\! \frac{{2a}}{p}\left\| u \right\|_p^p\!-\! \int_{\mathbb{R}^{N}}\bar{u}(\Omega \cdot L) u d x
\!\leq\! \displaystyle C^{*}(\Omega) \|u\|_{\dot{\Sigma}}^2 \!\leq\! \nu r C^{*}(\Omega).  
\end{align}
Hence by \eqref{2.4} and \eqref{2.5}, we have \eqref{2.3} holds provided $c<c_0=c_0(r,a,\Omega)$, where
\begin{align}\label{cbound2.5}
c_0:=\!\min\Big\{\frac{(1\!-\!|\Omega|) r}{4N},\Big[\frac{p(1\!-\!|\Omega|)^3}{{16(1\!+\!3|\Omega|) a\mathcal{C}^p_{N,p}  r^{\frac{p\delta_p-2}{2}}}} \Big]^{\frac{2}{p(1-\delta_p)}}, \Big[\frac{1\!-\!{|\Omega|}^2}{{2(1\!+\!3|\Omega|)a\mathcal{C}^p_{N,p}  r^{\frac{p\delta_p-2}{2}}}} \Big]^{\frac{2}{p(1-\delta_p)}}\Big\}.
\end{align}
\end{proof}

We also need the following Pohozaev identity.

\begin{proposition}  \label{pohozave}
Let $a,\lambda\!\in \R$, $N\!=\!2,3$, $2\!<\!p\!\leq\!2^*$ and $0\!<\!|\Omega|\!<\!1$. If $v\in \Sigma$ weakly solves
\begin{align}\label{Poeq3.6}
-\frac{1}{2} \Delta v+\frac{1}{2}|x|^2 v-(\Omega \cdot L)v -a|v|^{p-2} v=\lambda v,
\end{align}
then the Pohozaev identity
\begin{align}\label{Poho3.6}
Q(v):=\frac{1}{2}  {{{\| {\nabla v} \|}_2^2}}  \!-\! \frac{1}{2} {{{\| xv \|}_2^{2}}}\!-\!a\delta_p { {{\| v \|}_p^p}}\!=\!0
\end{align}
holds, where $\delta_p=\frac{N(p-2)}{2p}$.
\end{proposition}

\begin{proof}
Multiply \eqref{Poeq3.6} by $x\cdot \nabla \bar{v}$, integrate by parts and take real parts, we obtain
\begin{align}\label{p3.4}
\frac{2\!-\!N}{4}  {{{\| {\nabla v} \|}_2^2}}  \!-\! \frac{N\!+\!2}{4}{{{\| xv \|}_2^{2}}}\!+\! \frac{N\lambda}{2} {{{\| v \|}_2^2}}\!+\! \frac{Na}{p} { {{\| v \|}_p^p}}\!-\!\text{Re} \int_{{\R^N}} [(\Omega \cdot L) v] (x\cdot \nabla \bar{v})\!=\!0.
\end{align}
To eliminate $\lambda$, we multiply \eqref{Poeq3.6} by $\bar{v}$ and get
\begin{align}\label{p3.5}
\frac{1}{2}  {{{\| {\nabla v} \|}
_2^2}}  \!+\! \frac{1}{2} {{{\| xv|}_2^{2}}}\!-\! \lambda {{{\| v \|}_2^2}}\!-\!a { {{\| v \|}_p^p}}\!-\!\int_{{\R^N}} \bar{v}(\Omega \cdot L) v \!=\!0.
\end{align}
Combine \eqref{p3.4} and \eqref{p3.5}, we have the following Pohozaev identity
\begin{align}\label{p3.6}
\frac{1}{2} {{{\| {\nabla v} \|}_2^2}}  \!-\! \frac{1}{2} {{{\| xv \|}_2^{2}}}\!-\!a\delta_p { {{\| v \|}_p^p}}\!-\!\text{Re} \int_{{\R^N}} [(\Omega \cdot L) v] (x\cdot \nabla \bar{v})\!-\!\frac{N}{2}\int_{{\R^N}} \bar{v}(\Omega \cdot L) v \!=\!0.
\end{align}
The facts $\text{Re} \int_{{\R^N}} [(\Omega \cdot L) v] (x\cdot \nabla \bar{v})=-\text{Re} \int_{{\R^N}} [(\Omega \cdot L) \bar{v}] (x\cdot \nabla {v})$ and
$$
\int_{{\R^N}} [(\Omega \cdot L) v] (x\cdot \nabla \bar{v})=\int_{{\R^N}} [(\Omega \cdot L) \bar{v}] (x\cdot \nabla {v})-N\int_{{\R^N}} \bar{v}(\Omega \cdot L) v
$$
imply that
$$\text{Re} \int_{{\R^N}} [(\Omega \cdot L) v] (x\cdot \nabla \bar{v})=-\frac{N}{2}\int_{{\R^N}} \bar{v}(\Omega \cdot L) v.$$
Therefore, \eqref{p3.6} can be reduced to
\begin{align*}
Q(v):=\frac{1}{2}  {{{\| {\nabla v} \|}_2^2}}  \!-\! \frac{1}{2} {{{\| xv \|}_2^{2}}}\!-\!a\delta_p { {{\| v \|}_p^p}}\!=\!0.
\end{align*}
\end{proof}

We now prove the existence of a local minimizer.
\begin{proof}[\bf Proof of Theorem \ref{the1.5}] First, we show the existence of a local minimizer. It is sufficient to prove ${\mathcal{M}_c^r}\neq\emptyset$. Let $\{u_n\}\subset S(c)\cap B(r)$ be a minimizing sequence for $m_c^r= \mathop {\inf }\limits_{u \in S(c)\cap B(r)} I(u)$, then $\{u_n\}$ is bounded in $\Sigma$. By the compactness of the embedding $\Sigma\hookrightarrow L^q(\R^N,\mathbb{C})$ for $q\in[2,2^*)$, see Lemma \ref{Compact}, there exists $u\in \Sigma$ such that
\[\left\{ {\begin{array}{*{20}{c}}
   {{u_n} \rightharpoonup {u}\text{ in } \Sigma},\\
   {{u_n} \to {u}\text{ in }{L^q({\R^N},\mathbb{C})}},   \\
   {{u_n} \to {u}\text{ a.e. in }{\R^N}}.  \\
\end{array}} \right.\]
Consequently, we have $u\in S(c)\cap B(r)$. Moreover, we deduce from \eqref{boud2.4} that the energy functional $I$ is weakly lower semi-continuous.
Therefore, we have
\begin{align*}
I(u)\leq \lim_{n\to \infty}I(u_n)=m_c^r \leq I(u),
\end{align*}
which gives $I(u)=m_c^r$ and ${{u_n}\to {u}\text{ in } \Sigma}$. This implies that any minimizing sequence for $m_c^r$ is precompact and ${\mathcal{M}_c^r}\neq\emptyset$. For any $u_c\in {\mathcal{M}_c^r}$, Proposition \ref{prop2.3} implies that $u_c \not\in S(c)\cap \partial B(r)$ as $u_c\in B(\nu r)$, where $\partial B(r)\!=\!\Big\{u\!\in\! \Sigma: \|u\|_{\dot{\Sigma}}^2\!=\!r\Big\}$. Then $u_c$ is indeed a critical point of $I|_{S(c)}$. So, there exists a Lagrange multiplier $\omega_{c}\in \R$ such that $(u_{c},\omega_{c})$ is a couple of weak solution to problem \eqref{1.3}.

Next, we estimate the bound of the Lagrange multiplier $\omega_{c}$. Notice that the pure point spectrum of the harmonic oscillator is
\[\sigma_p(-\Delta+|x|^2)=\{\lambda_k=N+2k: k\in \mathbb{N}\}\]
and the corresponding eigenfunctions are given by Hermite functions (denoted by $\psi_k$, associated to $\lambda_k$), which form an orthonormal basis of $L^2(\R^N,\R)$ (see \cite{a}). Let $\psi_0 \in S(1)$ be an eigenfunction with respect to the first eigenvalue $\lambda_0=N$ and $\psi=\sqrt{c}\psi_0\in S(c)$. Then $\psi\in B(r)$ if $c\leq\frac{r}{N}$. As $\psi$ is real valued, we have $\int_{\mathbb{R}^{N}}\bar{\psi}(\Omega \cdot L) \psi d x=0$ and 
\begin{align}\label{3.2}
m_c^r\leq I(\psi)\!=\! \displaystyle\frac{1}{2}\|\psi\|_{\dot{\Sigma}}^2 \displaystyle\!-\!\frac{{2a}}{p}\left\| \psi \right\|_p^p\!-\! \int_{\mathbb{R}^{N}}\bar{\psi}(\Omega \cdot L) \psi d x \!<\! \displaystyle \frac{1}{2} \|\psi\|_{\dot{\Sigma}}^2 \!=\! \frac{1}{2}Nc.
\end{align}
Since $(u_{c},\omega_{c})\in {\mathcal{M}_c^r}\times \R$ weakly solves problem \eqref{1.3}, we learn from \eqref{3.2} that  
\begin{align} \label{lagrangemutiplier}
\omega_c{{{\| { u_c} \|}_2^2}}&=\frac{1}{2}\|u_c\|_{\dot{\Sigma}}^2 \displaystyle\!-\! a { {\| u_c \|_p^p}}\!-\! \int_{\mathbb{R}^{N}}\bar{u}_c(\Omega \cdot L) u_c d x \nonumber \\
&=I(u_c)+ \frac{a(2-p)}{p} { {{\| u_c \|}_p^p}}
<I(u_c)=m_c^r <\frac{1}{2}Nc, 
\end{align}
which implies that $\omega_c<\frac{N}{2}$. On the other hand, by \eqref{equ2.2} and \eqref{boud2.4}, we have
\begin{align*}
\omega_c {{{\| { u_c} \|}_2^2}}&=\frac{1}{2}\|u_c\|_{\dot{\Sigma}}^2 \displaystyle\!-\! a { {{\| u_c \|}_p^p}}\!-\! \int_{\mathbb{R}^{N}}\bar{u}_c(\Omega \cdot L) u_c d x \\
&\geq \displaystyle C_{*}(\Omega) \|u_c\|_{\dot{\Sigma}}^2-a \mathcal{C}^p_{N,p} \left\|\nabla u_c\right\|_{2}^{p\delta_p} c^{\frac{p(1-\delta_p)}{2}}\geq \displaystyle C_{*}(\Omega) \|u_c\|_{\dot{\Sigma}}^2 -a \mathcal{C}^p_{N,p} \|u_c\|_{\dot{\Sigma}}^{p\delta_p} c^{\frac{p(1-\delta_p)}{2}} \nonumber\\
&=\displaystyle\|u_c\|_{\dot{\Sigma}}^2 \Big(C_{*}(\Omega)-a \mathcal{C}^p_{N,p} \|u_c\|_{\dot{\Sigma}}^{p\delta_p-2} c^{\frac{p(1-\delta_p)}{2}}\Big)\geq \displaystyle Nc\Big(C_{*}(\Omega)-a \mathcal{C}^p_{N,p} r^{\frac{p\delta_p-2}{2}} c^{\frac{p(1-\delta_p)}{2}}\Big),
\end{align*} 
which implies that $\omega_c\geq N\Big(\frac{1-{|\Omega|}^2}{2(1+3|\Omega|)}-a \mathcal{C}^p_{N,p} r^{\frac{p\delta_p-2}{2}} c^{\frac{p(1-\delta_p)}{2}}\Big)>0$ as $c<c_0$, see \eqref{cbound2.5}.

Finally, we show that
\[\mathop {\sup }\limits_{u \in {\mathcal{M}_c^r}} \left\| {u - l_0{\psi _0}} \right\|_{ \Sigma }^2=O(c+c^{\frac{p(1-\delta_p)}{2}}) .\]
For any ${u \in {\mathcal{M}_c^r}}$, we rewrite $u=u_1+iu_2$, it results to
\[u=\sum_{k=0}^\infty {\big(\int_{{\R^N}} {u_1\psi_k}\big)\psi_k}+i\sum_{k=0}^\infty {\big(\int_{{\R^N}} {u_2\psi_k}\big)\psi_k} =\sum_{k=0}^\infty {l_k\psi_k} \text{ with } l_k=\int_{{\R^N}} {u\psi_k},\]
where $u_1$ is the real part and  $u_2$ is the imaginary part of $u$, $\{\psi_k\}$ is an orthonormal basis of $L^2(\R^N,\R)$. Thus
\[c={\|u\|_2^2}=\sum_{k=0}^\infty {l_k\bar{l}_k\int_{{\R^N}} {|\psi_k|^2}}=\sum_{k=0}^\infty {{|l_k|}^2},\]
where $\bar{l}_k$ is the conjugate of $l_k$. Since ${u \in {\mathcal{M}_c^r}}\subset B(r)$, by \eqref{equ2.2} and \eqref{boud2.4}, we have
\begin{align*}\nonumber
\frac{Nc}{2}>I(u)&= \displaystyle\frac{1}{2}\|u\|_{\dot{\Sigma}}^2 - \frac{{2a}}{p} { {{\| u \|}_p^p}}- \int_{\mathbb{R}^{N}}\bar{u}(\Omega \cdot L) u d x \\
&\geq \displaystyle C_{*}(\Omega) \|u\|_{\dot{\Sigma}}^2-\frac{{2a}}{p} \mathcal{C}^p_{N,p} \left\|\nabla u\right\|_{2}^{p\delta_p} c^{\frac{p(1-\delta_p)}{2}}
\geq \displaystyle C_{*}(\Omega) \|u\|_{\dot{\Sigma}}^2 -\frac{{2a}}{p} \mathcal{C}^p_{N,p} \|u\|_{\dot{\Sigma}}^{p\delta_p} c^{\frac{p(1-\delta_p)}{2}}  \\
&\geq \displaystyle C_{*}(\Omega) \|u\|_{\dot{\Sigma}}^2-\frac{{2a}}{p} \mathcal{C}^p_{N,p} r^{\frac{p\delta_p}{2}} c^{\frac{p(1-\delta_p)}{2}}
= \displaystyle C_{*}(\Omega)\sum_{k=0}^\infty {\lambda_k{{|l_k|}^2}}-\frac{{2a}}{p} \mathcal{C}^p_{N,p} r^{\frac{p\delta_p}{2}} c^{\frac{p(1-\delta_p)}{2}}
\end{align*}
which implies that
\[ N\sum_{k=1}^\infty {{|l_k|}^2}\leq \sum_{k=1}^\infty {\lambda_k {{|l_k|}^2} }\leq\sum_{k=0}^\infty {\lambda_k {{|l_k|}^2} }\leq \frac{Nc}{2C_{*}(\Omega)}+\frac{2a}{p}
 \frac{ \mathcal{C}^p_{N,p}  }{C_{*}(\Omega)}r^{\frac{p\delta_p}{2}} c^{\frac{p(1-\delta_p)}{2}} \]  
by using \eqref{3.2}. Thus, we have
\[\left\| {u - l_0{\psi _0}} \right\|_{\dot \Sigma }^2 = \left\| \sum_{k=1}^\infty l_k \psi_k\right\|_{\dot \Sigma }^2=\sum_{k=1}^\infty\lambda_k {{|l_k|}^2}\leq \frac{Nc}{2C_{*}(\Omega)}+\frac{2a}{p}
 \frac{ \mathcal{C}^p_{N,p}  }{C_{*}(\Omega)}r^{\frac{p\delta_p}{2}} c^{\frac{p(1-\delta_p)}{2}}\]
and
\[\left\| {u - l_0{\psi _0}} \right\|_2^2 = \left\| \sum_{k=1}^\infty l_k \psi_k\right\|_2^2=\sum_{k=1}^\infty {{|l_k|}^2}\leq \frac{c}{2C_{*}(\Omega)}+\frac{2a}{pN}
 \frac{ \mathcal{C}^p_{N,p}  }{C_{*}(\Omega)}r^{\frac{p\delta_p}{2}} c^{\frac{p(1-\delta_p)}{2}}.\]
Then, it follows from \eqref{consbound} that $\mathop {\sup }\limits_{u \in {\mathcal{M}_c^r}} \left\| {u - l_0{\psi _0}} \right\|_{ \Sigma }^2 \leq (N\!+\!1)\Big[\frac{1+3|\Omega|}{1-{|\Omega|}^2}c
+\frac{4(1+3|\Omega|)a \mathcal{C}^p_{N,p}}{pN(1-{|\Omega|}^2)}
  r^{\frac{p\delta_p}{2}} c^{\frac{p(1-\delta_p)}{2}}\Big]$. So we have $\mathop {\sup }\limits_{u \in {\mathcal{M}_c^r}} \left\| {u - l_0{\psi _0}} \right\|_{ \Sigma }^2=O(c+c^{\frac{p(1-\delta_p)}{2}})$. 
\end{proof}

Next, we show that $u_c $ is a normalized ground state if $c>0$ is sufficiently small. We also concern the asymptotic behavior of $u_c$ obtained by Theorem \ref{the1.5} as $c\to 0^+$.

\begin{proof}[\bf Proof of Theorem \ref{Asymtotic}]

This is motivated by \cite{BNlv}. On the contrary, we assume that there exists a $v\in S(c)$ such that
\[ I^{'}|_{S(c)}(v)=0 \text{ and } I(v)<m_c^r.\]
Since $I^{'}|_{S(c)}(v)=0$, then $v$ satisfies
\begin{align}\label{3.3}
\left(-\frac{1}{2} \Delta+\frac{1}{2}|x|^2-(\Omega \cdot L)\right) v-a|v|^{p-2} v=\lambda v, ~~~~~~~~x\in \R^N
\end{align}
for some $\lambda\in \R$.
It follows from Proposition \ref{pohozave} that $Q(v):=\frac{1}{2}  {{{\| {\nabla v} \|}_2^2}}  \!-\! \frac{1}{2} {{{\| xv \|}_2^{2}}}\!-\!a\delta_p { {{\| v \|}_p^p}}\!=\!0$. Therefore, we have
\[I(v)=\Big(\frac{1}{2}-\frac{1}{p\delta_p} \Big) {{{\| {\nabla v} \|}_2^2}}  \!+\! \Big(\frac{1}{2}+\frac{1}{p\delta_p} \Big) {{{\| xv \|}_2^{2}}}\!-\!\int_{{\R^N}} \bar{v}(\Omega \cdot L) v .\]   
Since $0<|\Omega|<\sqrt{1-(\frac{2}{p\delta_p})^2}$, we can choose $\varepsilon_1\in (\frac{p\delta_p}{p\delta_p+2}|\Omega|^{2},1-\frac{2}{p\delta_p})$ and denote
\begin{align}  \label{cOmega}
C_{1}(\Omega):=\frac{1}{2}-\frac{1}{p\delta_p}-\frac{\varepsilon_1}{2},~~~~ C_{2}(\Omega):=\frac{1}{2}+\frac{1}{p\delta_p}-\frac{|\Omega|^{2}}{2\varepsilon_1},
~~~~\mathcal{C}_\Omega:=\min\{C_{1}(\Omega),C_{2}(\Omega) \}.
\end{align}
It's easy to see that $C_{1}(\Omega)\!>\!0$, $C_{2}(\Omega)\!>\!0$. Applying inequality \eqref{Gn1.2} with $\varepsilon\!=\!\varepsilon_1$, we have
\begin{align*}
I(v)
\!\geq\! C_{1}(\Omega)\|\nabla v \|_{2}^{2}\!+\!C_{2}(\Omega) \|x v \|_{2}^{2} \!\geq\! \min\{C_{1}(\Omega),C_{2}(\Omega) \} \|v\|_{\dot{\Sigma}}^2=\mathcal{C}_\Omega \|v\|_{\dot{\Sigma}}^2.
\end{align*}
Thus, we deduce from \eqref{3.2} that
\[
\mathcal{C}_\Omega \|v\|_{\dot{\Sigma}}^2 \leq I(v) <m_c^r<\frac{Nc}{2}\to 0\ \text{as}\ c\to 0.
\]
If $c$ is sufficiently small, we have $v\in B(r)$ and $I(v)\geq m_c^r$, which contradicts to $I(v)<m_c^r$.

Next, we show that $u_{c} \to 0$ in $\dot{\Sigma}$ as $c\to 0^+$. Since $(u_{c},\omega_{c})\in {\mathcal{M}_c^r}\times \R$ weakly solves \eqref{1.3}, Proposition \ref{pohozave} indicates that $Q(u_{c})=0$.
Then, $I(u_{c})$ can be rewrite as $$I(u_{c})=\Big(\frac{1}{2}-\frac{1}{p\delta_p} \Big) {{{\| {\nabla u_{c}} \|}_2^2}}  \!+\! \Big(\frac{1}{2}+\frac{1}{p\delta_p} \Big) {{{\| xu_{c} \|}_2^{2}}}\!-\!\int_{{\R^N}} \bar{u}_{c}(\Omega \cdot L) u_{c}.$$
Applying inequality \eqref{Gn1.2} with $\varepsilon\!=\!\varepsilon_1$, we also have
\begin{align*}
I({u_{c}})
\!\geq\! C_{1}(\Omega)\|\nabla {u_{c}} \|_{2}^{2}\!+\!C_{2}(\Omega) \|x {u_{c}} \|_{2}^{2} \!\geq\! \min\{C_{1}(\Omega),C_{2}(\Omega) \} \|{u_{c}}\|_{\dot{\Sigma}}^2=\mathcal{C}_\Omega \|{u_{c}}\|_{\dot{\Sigma}}^2.
\end{align*}
Therefore, it holds that $\mathcal{C}_\Omega \|{u_{c}}\|_{\dot{\Sigma}}^2 \leq I({u_{c}})=m_c^r<\frac{Nc}{2}\to 0$ as $c\to 0^+$.

Let $v_c:=\frac{u_{c}}{\|{u_{c}}\|_2}=\frac{u_{c}}{\sqrt{c}}$, then we have the following estimates
\begin{align} \label{UpBod}
\mathcal{C}_\Omega \|{v_{c}}\|_{\dot{\Sigma}}^2 \leq \frac{I({u_{c}})}{\|{u_{c}}\|^2_2}=\frac{m_c^r}{c}
<\frac{N}{2}.
\end{align}
By using \eqref{equ2.2} and \eqref{UpBod}, we deduce that
\[
 0<\frac{\|{u_{c}}\|^p_p}{\|{u_{c}}\|^2_2}\leq  \frac{  \mathcal{C}^p_{N,p} \left\|\nabla u_{c} \right\|_{2}^{p\delta_p} \left\|u_{c}\right\|_{2}^{p(1-\delta_p)} }{ \|{u_{c}}\|^2_2 }=
\mathcal{C}^p_{N,p} \left\|\nabla v_{c} \right\|_{2}^{p\delta_p} \|{u_{c}}\|^{p-2}_2 \leq C c^{\frac{p-2}{2}} \to 0
\]
as $c\to0^+$, where $C=C(p,N,|\Omega|)>0$ is some constant. From Proposition \ref{pohozave}, we have
$$0=\frac{Q(u_{c})}{ \|{u_{c}}\|^2_2 }=\frac{1}{2}  {{{\| {\nabla v_{c}} \|}_2^2}}  \!-\! \frac{1}{2} {{{\| xv_{c} \|}_2^{2}}}\!-\!a\delta_p \frac{\|{u_{c}}\|^p_p}{\|{u_{c}}\|^2_2},$$
which gives
$\mathop {\lim }\limits_{c\to 0^+}{{{\| {\nabla v_{c}} \|}_2^2}}=\mathop {\lim }\limits_{c\to 0^+}{{{\| {x v_{c}} \|}_2^2}}$. Since $  N\Big(\frac{1-{|\Omega|}^2}{2(1+3|\Omega|)}-a \mathcal{C}^p_{N,p} r^{\frac{p\delta_p-2}{2}} c^{\frac{p(1-\delta_p)}{2}}\Big)\!\leq\!\omega_c\!<\!
\frac{N}{2}$, then there exists an $\omega\in[ \frac{(1-{|\Omega|}^2)N}{2(1+3|\Omega|)}, \frac{N}{2}]$ such that $\mathop {\lim }\limits_{c\to 0^+}\omega_c=\omega$ as $c\to0^+$. By these facts and \eqref{lagrangemutiplier}, we have
\begin{align*}
\mathop {\lim }\limits_{c\to 0^+}\omega_c =\mathop {\lim }\limits_{c\to 0^+} \Big[ \frac{1}{2}\|v_c\|_{\dot{\Sigma}}^2 \!-\! \int_{\mathbb{R}^{N}}\bar{v}_c(\Omega \cdot L) v_c d x\!-\! a \frac{\|{u_{c}}\|^p_p}{\|{u_{c}}\|^2_2}\Big]=\mathop {\lim }\limits_{c\to 0^+} \Big[ \frac{1}{2}\|v_c\|_{\dot{\Sigma}}^2 \!-\! \int_{\mathbb{R}^{N}}\bar{v}_c(\Omega \cdot L) v_c d x\Big], \\
\mathop {\lim }\limits_{c\to 0^+}\frac{m_c^r}{c}
=\mathop {\lim }\limits_{c\to 0^+} \Big[\frac{1}{2}\|v_{c}\|_{\dot{\Sigma}}^2 \!-\! \int_{\mathbb{R}^{N}}\bar{v}_{c}(\Omega \cdot L) v_{c} d x\!-\! \frac{{2a}}{p}\frac{\|{u_{c}}\|^p_p}{\|{u_{c}}\|^2_2}\Big]=\mathop {\lim }\limits_{c\to 0^+} \Big[\frac{1}{2}\|v_{c}\|_{\dot{\Sigma}}^2 \!-\! \int_{\mathbb{R}^{N}}\bar{v}_{c}(\Omega \cdot L) v_{c} d x\Big].
\end{align*}
Finally, we deduce that $ \mathop {\lim }\limits_{c\to 0^+}\frac{m_c^r}{c}=\mathop {\lim }\limits_{c\to 0^+}\omega_c=\omega$ and
$$
\mathop {\lim }\limits_{c\to 0^+}  \frac{\|\nabla u_{c}\|_2^2\!-\! \int_{\mathbb{R}^{N}}\bar{u}_{c}(\Omega \cdot L) u_{c} d x}{c} =\mathop {\lim }\limits_{c\to 0^+}  \frac{\|x u_{c}\|_2^2\!-\! \int_{\mathbb{R}^{N}}\bar{u}_{c}(\Omega \cdot L) u_{c} d x}{c} =\omega.$$
\end{proof}

At the end of this Section, we prove Theorem \ref{stability}, i.e. the stability of ${\mathcal{M}_c^r}$.

\begin{proof}[\bf Proof of Theorem \ref{stability}]
Just suppose that there exists an $\varepsilon_0>0$, a sequence of initial data $\{u_n^0\}\subset \Sigma$ and a sequence $\{t_n\}\subset \R^+$ such that the unique solution $u_n$ of problem \eqref{1.1} with initial data $u_n(0,\cdot)=u_n^0(\cdot)$ satisfies
\[\text{dist}_\Sigma(u_n^0,{\mathcal{M}_c^r})< \frac{1}{n} \text{ and } \text{ dist}_\Sigma\Big(u_n(t_n,\cdot),{\mathcal{M}_c^r}\Big)\geq \varepsilon_0.\]
Without loss of generality, we may assume that $\{u_n^0\}\subset S(c)$.
Since $\text{dist}_\Sigma(u_n^0,{\mathcal{M}_c^r})\to 0$ as $n\to\infty$, the conservation laws of the energy and mass imply that
$\{u_n(t_n,\cdot)\}$ is a minimizing sequence for $m_c^r= \mathop {\inf }\limits_{u \in S(c)\cap B(r)} I(u)$ provided $\{u_n(t_n,\cdot)\}\subset B(r)$. Indeed, if $\{u_n(t_n,\cdot)\}\subset (\Sigma \setminus B(r))$, then by the continuity there exists ${\bar{t}}_n\in [0,t_n)$ such that $\{u_n({\bar{t}}_n,\cdot)\}\subset {\partial B(r)}$, where
 $\partial B(r)\!=\!\Big\{u\!\in\! \Sigma: \|u\|_{\dot{\Sigma}}^2\!=\!r\Big\}$.
Hence by Proposition \ref{prop2.3},
\[I(u_n({\bar{t}}_n,\cdot))\geq \mathop {\inf }\limits_{u \in S(c)\cap {\partial B(r)}} I(u)> \mathop {\inf }\limits_{u \in S(c)\cap B(\nu r)} I(u)=\mathop {\inf }\limits_{u \in S(c)\cap B(r)} I(u)=m_c^r,\]
which is a contradiction. Therefore, $\{u_n(t_n,\cdot)\}$ is a minimizing sequence for $m_c^r$.
Then there exists $v_0 \in {\mathcal{M}_c^r}$ such that $u_n(t_n,\cdot)\to v_0$ in $\Sigma$, which contradicts to \[\text{dist}_\Sigma\Big(u_n(t_n,\cdot),{\mathcal{M}_c^r}\Big)\geq \varepsilon_0.\]
\end{proof}

\vspace{.25cm}
\section{Proof of Theorems \ref{the1.8}}
In this section, we prove Theorem \ref{the1.8}, i.e. the existence of a mountain pass solution. Let us fix $u_c\in {\mathcal{M}_c^r}$ and
$v_c(x)=l^{\frac{N}{2}}u_c(lx)$ for $l>>1$ such that $v_c\in S(c)\setminus B(r)$ and $I(v_c)<0$ (${\mathcal{M}_c^r}$ is defined in \eqref{1.9++}). First, we introduce a min-max class
\begin{align}\label{4.2}
{ \Gamma}(c): = \{ g\in C([0,1], S(c)) : g(0)=u_c \text{ and } g(1)=v_c\}
\end{align}
and a min-max value
\begin{align}\label{4.3}
{\gamma}(c): = \mathop {\inf }\limits_{ g \in {\Gamma}(c)} \mathop {\max }\limits_{0 \le t \le 1}  I(g(t)).
\end{align}
Notice that ${ \Gamma}(c)\neq \emptyset$, for $g(t)=(1+tl-t)^{\frac{N}{2}}u_c(x+t(l-1)x) \in { \Gamma}(c)$. By \eqref{4.2} and Proposition \ref{prop2.3}, we have
\begin{align}\label{4.4}
{\gamma}(c)>\max\{I(u_c),I(v_c)\}>0.
\end{align}

Next, we introduce an auxiliary functional $\widetilde I:S(c) \!\times \! \R \!\to\! \R,~~~~(u,\theta ) \!\to\! I(\kappa (u,\theta ))$ for $\kappa (u,\theta ): =\! {e^{\frac{N}{2}\theta }}u({e^\theta }x)$. To be precise, we have
\begin{align*}
\dis \widetilde I(u,\theta)=I(\kappa (u,\theta ))= \frac{e^{2\theta }}{2} {{{\| {\nabla u} \|}_2^2}}  + \frac{1}{2e^{2\theta }} {{{\| xu \|}_2^2}}- \frac{{2a}}{p}e^{p\delta_p\theta} { {{\| u \|}_p^p}}-\int_{{\R^N}} \bar{u}(\Omega \cdot L) u.
\end{align*}
Define a set of paths
\begin{align}\label{4.1}
{\widetilde \Gamma}(c): = \{ \widetilde g\in C([0,1], S(c) \times \R ) : \widetilde g(0)=(u_c,0) \text{ and } \widetilde g(1)=(v_c,0)\}
\end{align}
and a minimax value
\[{\widetilde\gamma}(c): = \mathop {\inf }\limits_{\widetilde g \in {\widetilde \Gamma}(c)} \mathop {\max }\limits_{0 \le t \le 1} \widetilde I(\widetilde g(t)),\]
we claim that ${\widetilde\gamma}(c)={\gamma}(c)$. In fact, it follows immediately from the definition of ${\widetilde\gamma}(c)$ and ${\gamma}(c)$ along with the fact that the maps
\[\varphi :{\Gamma(c)} \to {\widetilde\Gamma(c)},\text{ } g \to \varphi (g): = (g,0)~~~~~~~~\mbox{and}~~~~~~~~\psi :{\widetilde\Gamma(c)} \to {\Gamma(c)},\text{ } \widetilde g \to \psi (\widetilde g): = \kappa  \circ \widetilde g\]
satisfy \[\widetilde I(\varphi (g)) = I(g) \text{ and } I(\psi (\widetilde g)) = \widetilde I(\widetilde g).\]

Denote $|r|_{\R}\!=\!|r|$ for $r\!\in \!\R$, $E:=\!\Sigma \!\times\! \R$ endowed with the norm $\left\|  \cdot  \right\|_E^2 \!=\! {\left\|  \cdot  \right\|_{\Sigma}^2} \!+\! \left|  \cdot  \right|_\R^2$ and ${E^ {-1} }$ the dual space of $E$. We give two useful Lemmas.

\begin{lemma} \label{lem4.2}(\cite{j}, Lemma 2.3)
Let $\varepsilon>0$. Suppose that ${\widetilde g_0} \in {\widetilde\Gamma(c)}$ satisfies \[\mathop {\max }\limits_{0 \le t \le 1} \widetilde I({\widetilde g_0}(t)) \le {\widetilde\gamma}(c) + \varepsilon .\]
Then there exists a pair of $({u_0},{\theta _0}) \in S(c) \times \R$ such that:

(1)~~~~$\widetilde I({u_0},{\theta _0}) \in [{\widetilde\gamma}(c) - \varepsilon ,{\widetilde\gamma}(c) + \varepsilon]$;

(2)~~~~$\mathop {\min }\limits_{0 \le t \le 1} {\left\| {({u_0},{\theta _0}) - {{\widetilde g}_0}(t)} \right\|_E} \le \sqrt \varepsilon$;

(3)~~~~${\left\| {{{\left. {{{\widetilde I}^{'}}} \right|}_{{S}(c) \times \R}}({u_0},{\theta _0})} \right\|_{{E^ {-1} }}} \le 2\sqrt \varepsilon
, \text{ }i.e.~~\left| {{{\left\langle {{{\widetilde I}^{'}}({u_0},{\theta _0}),z} \right\rangle }_{{E^ {-1} } \times E}}} \right| \le 2\sqrt \varepsilon  {\left\| z \right\|_E}$ holds, for all \[z \in {\widetilde T_{({u_0},{\theta _0})}}: = \{ ({z_1},{z_2}) \in E,{\left\langle {{u_0},{z_1}} \right\rangle _{{L^2}}} = 0\}. \]
\end{lemma}

\begin{lemma} \label{lem4.4}(\cite{bl2}, Lemma 3)
Let $I \in {C^1(\Sigma,\R)}$. If $\left\{ {{v_n}} \right\} \subset S(c)$ is bounded in $\Sigma$, then
\[{\left. {{I^{'}}} \right|_{S(c)}}\left( {{v_n}} \right)\to 0 {\text{ in }}\Sigma^{-1}\Longleftrightarrow {I^{'}}\left( {{v_n}} \right) -   \frac{1}{c}\langle {I^{'}}\left( {{v_n}} \right),{v_n}\rangle {v_n} \to 0 {\text{ in }}\Sigma^{-1} \text{ as }n\rightarrow\infty.\]
\end{lemma}

Then, we construct a special Palai-Smale sequence for ${\gamma}(c)$ defined by \eqref{4.3} and show the compactness of the corresponding Palai-Smale sequence.

\begin{proposition} \label{pro4.3}
Let $a \!>\!0$, $N\!=\!2,3$, $2\!+\!\frac{4}{N}\!<\!p\!<\!2^*$, $0\!<\!|\Omega|\!<\!1$ and $c\!<\!c_0$ for $c_0$ obtained by Theorem \ref{the1.5}. Then, there exists a sequence $\{ {v_n}\}\!\subset\! S(c)$ such that
\begin{equation}\label{4.5}
\begin{array}{rl}
\displaystyle
\left\{ {\begin{array}{*{20}{c}}
   {I({v_n}) \to {\gamma}(c),}  \\
   {{{\left. {{I^{'}}} \right|}_{S(c)}}({v_n}) \to 0,}  \\
   {Q({v_n}) \to 0}  \\
\end{array}} \right.
\end{array}
\end{equation}
as $n \!\to\! +\infty$, where $Q(v_n)=\frac{1}{2}  {{{\| {\nabla v_n} \|}_2^2}}  \!-\! \frac{1}{2} {{{\| xv_n \|}_2^{2}}}\!-\!a\delta_p { {{\| v_n \|}_p^p}}$.
\end{proposition}

\begin{proof}
By the definition of ${{\gamma}(c)}$, there exists a ${g_n} \in {\Gamma(c)}$ such that
\[ {\gamma}(c) \le \mathop {\max }\limits_{0 \le t \le 1} I({g_n}(t)) \le {\gamma}(c) + \frac{1}{n},~~~~~~~~\forall n \in {\mathbb{N}^ + }.\]
Since ${\widetilde\gamma}(c) = {\gamma}(c),\text{ }{\widetilde g_n} = ({g_n},0) \in {\widetilde\Gamma(c)}$, we have
$\mathop {\max }\limits_{0 \le t \le 1} \widetilde I({\widetilde g_n}(t)) \le {\widetilde\gamma}(c) + \frac{1}{n}$.
Therefore, Lemma \ref{lem4.2} indicates the existence of a sequence $\{ ({u_n},{\theta_n})\}  \subset S(c) \times \R$ such that

$(i)~~~~\widetilde I({u_n},{\theta _n}) \in [{\gamma}(c) - \frac{1}{n},{\gamma}(c) + \frac{1}{n}]$;

$(ii)~~~~\mathop {\min }\limits_{0 \le t \le 1} {\left\| {({u_n},{\theta _n}) - ({{ g}_n}(t)},0) \right\|_E} \le \sqrt \frac{1}{n}$;

$(iii)~~~~{\left\| {{{\left. {{{\widetilde I}^{'}}} \right|}_{S(c) \times \R}}({u_n},{\theta _n})} \right\|_{{E^ {-1} }}} \le 2\sqrt \frac{1}{n}
,\text{ } i.e.~~\left| {{{\left\langle {{{\widetilde I}^{'}}({u_n},{\theta _n}),z} \right\rangle }_{{E^ {-1} } \times E}}} \right| \le 2\sqrt \frac{1}{n}  {\left\| z \right\|_E}$ holds for all \[z \in {\widetilde T_{({u_n},{\theta _n})}}: = \{ ({z_1},{z_2}) \in E,\text{ }{\left\langle {{u_n},{z_1}} \right\rangle _{{L^2}}} = 0\}. \]
Let ${v_n} = \kappa ({u_n},{\theta _n})$, $\forall n \in {\mathbb{N}^ + }$, then we prove that $\{ {v_n}\}\subset S(c)$ satisfies \eqref{4.5}. Firstly, from $(i)$ and the fact that ${I({v_n}) = I(\kappa ({u_n},{\theta _n})) = \widetilde I({u_n},{\theta _n})}$, we have $I({v_n}) \to {\gamma}(c)$ as $n \to +\infty$. Secondly, direct calculation implies that
\begin{equation}\label {4.9}
\begin{array}{rl}
\displaystyle
2Q({v_n})
&=\displaystyle{{{\| {\nabla v_n} \|}_2^2}}  -  {{{\| xv_n \|}_2^{2}}}-2a\delta_p { {{\| v_n \|}_p^p}} \\
&= \displaystyle{e^{2\theta_n }} {{{\| {\nabla u_n} \|}_2^2}}-{e^{-2\theta_n }} {{{\| xu_n \|}_2^{2}}}-2a\delta_p{e}^{p\delta_p\theta_n} { {{\| u_n \|}_p^p}}\\
&=\left\langle {{{\widetilde I}^{'}}({u_n},{\theta _n}),(0,1)} \right\rangle.\\
\end{array}
\end{equation}
Thus $(iii)$ yields $Q({v_n}) \to 0 \text{ as }n \to \infty,  \text{ for } (0,1) \in {\widetilde T_{({u_n},{\theta _n})}}.$
Finally, we prove that $${{{\left. {{I'}} \right|}_{S(c)}}({v_n}) \to 0 \text{ as }n \to \infty }.$$  We claim that for $n \in\mathbb{ N}$ sufficiently large, it holds that
$$\left| {\left\langle {{I'}({v_n}),\eta } \right\rangle } \right| \le \frac{2\sqrt2}{{\sqrt n }}{\left\| \eta  \right\|}~~~~,\forall \eta  \in {T_{{v_n}}}= \{ \eta \in \Sigma,\text{ }{\left\langle {{v_n},\eta } \right\rangle _{{L^2}}} = 0\}.$$
In fact, for any $\eta  \in {T_{{v_n}}}$, let $\widetilde\eta  = \kappa (\eta , - {\theta _n})$, we have
\begin{equation}\label{4.10}
\begin{array}{rl}
\displaystyle
\left\langle {{I^{'}}(v_n),\eta } \right\rangle= \left\langle {{{\widetilde I}^{'}}(u_n,\theta _n),(\widetilde \eta ,0)} \right\rangle.\\
\end{array}
\end{equation}
Since $\dis\int_{{\R^N}} {{u_n}\widetilde\eta }  = \dis\int_{{\R^N}} {{v_n}\eta }$, we obtain $(\widetilde\eta ,0) \in {\widetilde T_{({u_n},{\theta _n})}} \Leftrightarrow \eta  \in {T_{{v_n}}}$. It follows from $(ii)$ that
\[\left| {{\theta _n}} \right| = \left| {{\theta _n} - 0} \right| \le \mathop {\min }\limits_{0 \le t \le 1} {\left\| {({u_n},{\theta _n}) -({g_n}(t),0)} \right\|_E} \le \frac{1}{{\sqrt n }}.\]
Consequently, for $n$ large enough, we have
\[\left\| {(\widetilde\eta ,0)} \right\|_E^2 = {\left\| {\widetilde\eta } \right\|_\Sigma^2} =  {{{\| \eta  \|}_2^2}}  + {e^{ - 2{\theta _n}}} {{{\| {\nabla \eta } \|}_2^2}}+{e^{ 2{\theta _n}}}\dis  {{{\| {x \eta} \|}_2^2}}  \le 2{\left\| \eta  \right\|_\Sigma^2}.\]
Thus, $(iii)$ implies that
\[\left| {\left\langle {{I^{'}}({v_n}),\eta } \right\rangle } \right| = \left\langle {{{\widetilde I}^{'}}({u_n},{\theta _n}),(\widetilde\eta ,0)} \right\rangle  \le \frac{2}{{\sqrt n }}\left\| {(\widetilde\eta ,0)} \right\|_E \le \frac{2\sqrt2}{{\sqrt n }}{\left\| \eta  \right\|_\Sigma}.\]
It results to
\[\left\| {{{\left. {{I^{'}}} \right|}_{S(c)}}({v_n})} \right\|_{\Sigma^{-1}} = \mathop {\sup }\limits_{\eta  \in {T_{{v_n}}},\left\| \eta \right\| \le 1} \left| {\left\langle {{I^{'}}({v_n}),\eta } \right\rangle } \right| \le \frac{2\sqrt2}{{\sqrt n }} \to 0 \text{ as }n \to +\infty .\]
\end{proof}

\begin{proposition}\label{pro4.5}
Assume that $a \!>\!0$, $N\!=\!2,3$, $2\!+\!\frac{4}{N}\!<\!p\!<\!2^*$,  $0\!<\!|\Omega|\!<\!\sqrt{1-(\frac{2}{p\delta_p})^2}$ and $c\!<\!c_0$ for $c_0$ obtained by Theorem \ref{the1.5}. Let
$\{ {v_n}\}\subset S(c)$ be a sequence such that
\begin{equation}\label {4.12}
\begin{array}{rl}
\displaystyle
\left\{ {\begin{array}{*{20}{c}}
   {I({v_n}) \to {\gamma}(c),}  \\
   {{{\left. {{I^{'}}} \right|}_{S(c)}}({v_n}) \to 0,}  \\
   {Q({v_n}) \to 0} \\
\end{array}} \right.
\end{array}
\end{equation}
as $n \to +\infty$. Then there exist a ${v} \in \Sigma$, a sequence $\{\omega_n\}\subset \R$ and a $\widetilde{\omega}\in \R$ such that\\
$(i) {v_n} \to {v} \text{ in }\Sigma$, up to a subsequence, as $n \to +\infty$;\\
$(ii) \omega_n \to \widetilde{\omega}\text{ in } \R$, up to a subsequence, as $n \to +\infty$;\\
$(iii)\text{Re} \big[ {{I^{'}}({v_n})- \omega_n  {v_n} }  \big] \to 0\text{ in } \Sigma^{-1}$, up to a subsequence, as $n \to +\infty$;\\
$(iv) \text{Re} \big[ {{I^{'}}({v})-\widetilde{\omega} {v} }   \big]= 0 \text{ in } \Sigma^{-1}$.
\end{proposition}

\begin{proof}
We first show that $\{v_n\}$ is bounded in $\Sigma$. Notice that $p\delta_p>2$ as $p\in(\bar{p},2^*)$. By using $Q({v_n}) \to 0$, we have
\[I({v_n})=\Big(\frac{1}{2}-\frac{1}{p\delta_p} \Big) {{{\| {\nabla {v_n}} \|}_2^2}}  \!+\! \Big(\frac{1}{2}+\frac{1}{p\delta_p} \Big) {{{\| x{v_n} \|}_2^{2}}}\!-\!\int_{{\R^N}} \bar{{v}}_n(\Omega \cdot L) {v_n}+o_n(1) \leq {\gamma}(c) +1.\]
Since $0<|\Omega|<\sqrt{1-(\frac{2}{p\delta_p})^2}$, we can choose $\varepsilon_1\in (\frac{p\delta_p}{p\delta_p+2}|\Omega|^{2},1-\frac{2}{p\delta_p})$ and denote
\begin{align*} 
C_{1}(\Omega):=\frac{1}{2}-\frac{1}{p\delta_p}-\frac{\varepsilon_1}{2}>0,~~~~ C_{2}(\Omega):=\frac{1}{2}+\frac{1}{p\delta_p}-\frac{|\Omega|^{2}}
{2\varepsilon_1}>0,~~~~\mathcal{C}_\Omega=\min\{C_{1}(\Omega),C_{2}(\Omega) \}.
\end{align*} 
Applying inequality \eqref{Gn1.2} with $\varepsilon\!=\!\varepsilon_1$, we have
\begin{align*}
{\gamma}(c) +1\!\geq\!I({v_n})
\!\geq\! C_{1}(\Omega)\|\nabla {v_n} \|_{2}^{2}\!+\!C_{2}(\Omega) \|x {v_n} \|_{2}^{2}+o_n(1) \!\geq\! \mathcal{C}_\Omega \|{v_n}\|_{\dot{\Sigma}}^2+o_n(1).
\end{align*}
Thus, $\{v_n\}$ is bounded in $\Sigma$. Then, up to a subsequence, there exists a ${v} \in \Sigma$ such that
\[\left\{ {\begin{array}{*{20}{c}}
   {{v_n} \rightharpoonup {v}\text{ in } \Sigma},       \\
   {{v_n} \to {v}\text{ in }{L^2}({\R^N},\mathbb{C})},  \\
   {{v_n} \to {v}\text{ in }{L^p}({\R^N},\mathbb{C})},  \\
   {{v_n} \to {v}\text{ a.e in }{\R^N}}.                \\
\end{array}} \right.\]
By Lemma \ref{lem4.4}, we know that
\[{\left. {{I^{'}}} \right|_{S(c)}}\left( {{v_n}} \right)\to 0 {\text{ in }}{\Sigma^{ - 1}}\Longleftrightarrow {I^{'}}\left( {{v_n}} \right) - \frac{1}{c}\langle {I^{'}}\left( {{v_n}} \right),{v_n}\rangle {v_n} \to 0 {\text{ in }}{\Sigma^{ - 1}} \text{ as } n\rightarrow +\infty.\]
Therefore, we have $\text{Re} \left\langle {{I^{'}}({v_n}) - \frac{1}{c}\left\langle {{I^{'}}({v_n}),{v_n}} \right\rangle {v_n},\varphi} \right\rangle\to0$ for each $\varphi\in \Sigma$, that is
\begin{align}\label {4.13}
\text{Re} \Big[\frac{1}{2}  \int_{{\R^N}} \big({\nabla {v_n} \nabla \bar{\varphi} }  \!+\!  {|x|^2{v_n} \bar{\varphi} }\big)\!-\!\int_{{\R^N}} \bar{\varphi}(\Omega \cdot L) {v_n}
\!-\!a\int_{{\R^N}}|{v_n}|^{p-2} {v_n}\bar{\varphi}\!-\!\omega_n \int_{{\R^N}} {{v_n}\bar{\varphi} }\Big]\!\to\! 0,
\end{align}
where \begin{equation}\label {4.14}
\begin{array}{rl}
\displaystyle
{\omega _n}=\frac{1}{c}\langle {I^{'}}\left( {{v_n}} \right),{v_n}\rangle=\frac{1}{c}\Big(  \frac{1}{2} {{{\| {\nabla {v_n}} \|}_2^2}}  \!+\!\frac{1}{2} {{{\| x{v_n} \|}_2^{2}}}\!-\!\int_{{\R^N}} \bar{{v}}_n(\Omega \cdot L) {v_n}\!-\!a \|{v_n}\|^{p}_p \Big).
\end{array}
\end{equation}
Thus $(iii)$ is proved. By Lemma \ref{lem2.1} and inequality \eqref{Gn1.2}, each term in the right hand of \eqref{4.14} is bounded. So there exists $\widetilde{\omega} \in \R$ such that, up to a subsequence, $\omega_n \to \widetilde{\omega}$ as $n \to +\infty$.
Thus $(ii)$ is proved and $(iv)$ follows from $(iii)$.
By $(ii)$ $(iii)$ and $(iv)$ we have
\begin{equation}\label{4.15}
\begin{array}{rl}
\displaystyle
\text{Re} \left\langle {{I^{'}}({v_n}) - \widetilde \omega {v_n},{v_n} - v} \right\rangle  = o_n(1) \text{ and } \text{Re} \left\langle {{I^{'}}({v}) - \widetilde \omega{v},{v_n} - v} \right\rangle  = 0.
\end{array}
\end{equation}
We deduce from \eqref{4.15} and \eqref{boud2.4} that
\begin{align*}
o_n(1)=\frac{1}{2}\|({v_n} - v)\|_{\dot{\Sigma}}^2-\int_{\mathbb{R}^{N}}\overline{({v_n} - v)}(\Omega \cdot L) ({v_n} - v) d x\geq C_{*}(\Omega) \|({v_n} - v)\|_{\dot{\Sigma}}^2 .
\end{align*}
It results to ${v_n} \to {v}$ in $\dot{\Sigma}$ as $n \to \infty$. As ${v_n} \to {v}$ in ${L^2}({\R^N},\mathbb{C})$, we see that ${v_n} \to {v}$ in ${\Sigma}$ and $(i)$ is proved.
\end{proof}

\vspace{.25cm}
\begin{proof}[\bf Proof of Theorem \ref{the1.8}]
Propositions \ref{pro4.3}-\ref{pro4.5} guarantee the existence of a couple of weak solution $(\hat{u}_c,\hat{\omega}_c)\in \Sigma \times \R$ to problem \eqref{1.3} with $\left\| {\hat{u}_c} \right\|_2^2 = c$. By using \eqref{4.4}, we have
\[I(\hat{u}_c)=\gamma(c)>I(u_c)=m_c^r.\]
\end{proof}

\end{document}